\documentclass[letterpaper,12pt]{article}
\usepackage{amssymb,amsmath,graphicx}
\newtheorem{theorem}{Theorem}[section]
\newtheorem{definition}{Definition}[section]
\newtheorem{lemma}{Lemma}[section]

\newenvironment{proof}[1][Proof]{\begin{trivlist}
\item[\hskip \labelsep {\bfseries #1}]}{\end{trivlist}}
\begin{document}

\title{A Fast Summation Method for Oscillatory Lattice Sums}
\author{Ryan Denlinger
\thanks{Courant Institute of Mathematical Sciences,
         New York University, 251 Mercer Street,
         New York, NY 10012-1110.
{{\em email}: {\sf {ryand@cims.nyu.edu.}}}}
\and
Zydrunas Gimbutas
\thanks{National Institute of Standards and Technology,
         325 Broadway, Boulder, CO 80305.
{{\em email}: {\sf zydrunas.gimbutas@boulder.nist.gov.}}
Contributions by staff of NIST, an agency of the U.S. Government, 
are not subject to copyright within the United States.}
\and 
Leslie Greengard
\thanks{Courant Institute of Mathematical Sciences,
         New York University, 251 Mercer Street,
         New York, NY 10012-1110.
{{\em email}: {\sf {greengard@cims.nyu.edu.}}}}
\and
Vladimir Rokhlin
\thanks{Departments of Computer Science, Mathematics, and Physics,
Yale University, New Haven, CT 06511.
{{\em email}: {\sf rokhlin@cs.yale.edu.}}}
}
\date{\today}
\maketitle

\begin{abstract}
We present a fast summation method for lattice sums of the 
type which arise when solving wave scattering problems with periodic boundary
conditions. While there are a variety of effective algorithms in the 
literature for such calculations, the approach presented here is 
new and leads to a rigorous analysis of Wood's anomalies. These arise when
illuminating a grating at specific combinations of the angle of incidence 
and the frequency of the wave, for which the lattice sums diverge. 
They were discovered by Wood in 1902 as singularities in the spectral response.
The primary tools in our approach are the
Euler-Maclaurin formula and a steepest descent argument. The resulting
algorithm has super-algebraic convergence and requires only 
milliseconds of CPU time.

\end{abstract}

\section{Introduction}

A variety of problems in acoustics and electromagnetics require the 
solution of {\em quasiperiodic} scattering problems - that is,
time-harmonic wave scattering from a periodic structure with a well-defined 
unit cell
\cite{Petit1980,Wilcox}. 
For the sake of simplicity, we will focus on the scalar (acoustic) case.
In the three dimensional setting, we imagine that
an acoustic plane wave of the form 
$u^{in} = e^{i(k_1x + k_2y + k_3z)}$ 
impinges on a two-dimensional array of scatterers centered at 
$(nd_1,md_2)$ for $n,m \in \mathbb{Z}$ (Fig. \ref{scatfig}, left).
We denote the unit cell centered at the origin by 
$U = [-\frac{d_1}{2},\frac{d_1}{2}] \times [-\frac{d_2}{2}, \frac{d_2}{2}]
\times (-\infty,\infty)$
and the scatterer centered in the unit cell by $S$.
The incoming wave satisfies  the Helmholtz equation 
\begin{equation}
\Delta u + k^2 u = 0\, ,
\label{helmeq} 
\end{equation} 
where $k = \sqrt{k_1^2 + k_2^2 + k_3^2}$, and
{\em quasiperiodic} boundary conditions
on $U$, namely,
\begin{align*}
u^{in}(x+d_1,y,z) &= e^{i \alpha} u^{in}(x,y,z) \, ,\\
u^{in}(x,y+d_2,z) &= e^{i \beta} u^{in}(x,y,z) \, ,
\end{align*}
where $e^{i\alpha} = e^{ik_1d_1}$ and $e^{i \beta} = e^{ik_2d_2}$ are complex
(Bloch) phases. Note that for a normally incident wave, with 
$k_3=k$ and $k_1=k_2=0$, the boundary conditions reduce to simple
periodicity.
Assuming a ``sound-soft" obstacle $S$,
the scattered field $u^{scat}$ exterior to $S$ in the domain
$(x,y,z) \in U$ must satisfy
the Helmholtz equation (\ref{helmeq}) and the boundary conditions
\begin{align*}
u^{scat}(x+d_1,y,z) &= e^{i \alpha} u^{scat}(x,y,z) \, ,\\
u^{scat}(x,y+d_2,z) &= e^{i \beta} u^{scat}(x,y,z) \, ,\\
u^{scat}(x,y,z) &= - u^{in}(x,y,z) |_S .
\end{align*}

In the two dimensionsal case, 
the incoming acoustic plane wave takes the form 
$u^{in} = e^{i(k_1x + k_2y)} = 
e^{i(k \cos \psi x + k \sin \psi y)}$,
where $k = \sqrt{k_1^2 + k_2^2}$ and $\psi$ is the angle of incidence
of the incoming wave with respect to the $x$-axis.
The one-dimensional array of scatterers is assumed to be centered at 
$(nd,0)$ for $n \in \mathbb{Z}$ (Fig. \ref{scatfig}, right).
The unit cell will again denoted by
$U = [-\frac{d}{2},\frac{d}{2}] \times (-\infty,\infty)$
and the scatterer centered at the origin will again denoted by $S$.
The incoming wave satisfies the Helmholtz equation,
while the quasiperiodic condition
is now simply
\begin{align*}
u^{in}(x+d,y) &= e^{i \alpha} u^{in}(x,y) \, ,
\end{align*}
with Bloch phase $e^{i \alpha} = e^{ik_1d}$.
The scattered field $u^{scat}$, exterior to $S$ but within the domain
$U$, must satisfy
the Helmholtz equation (\ref{helmeq}) and the boundary conditions
\begin{align*}
u^{scat}(x+d,y) &= e^{i \alpha} u^{scat}(x,y) \, , \\
u^{scat}(x,y) &= - u^{in}(x,y) |_S .
\end{align*}

\begin{figure}[h]
\centering
\includegraphics[width=0.45\textwidth]{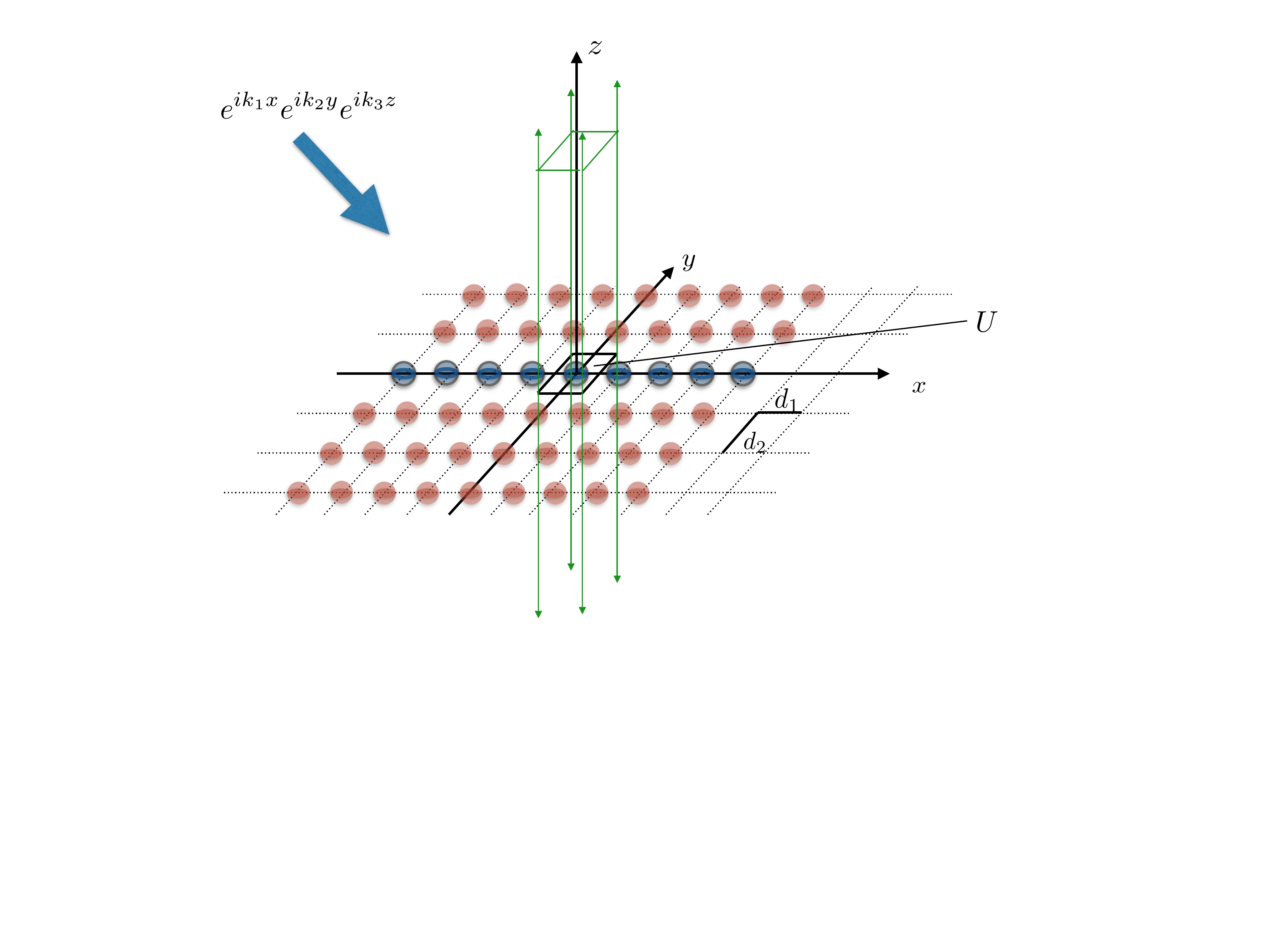}
\includegraphics[width=0.45\textwidth]{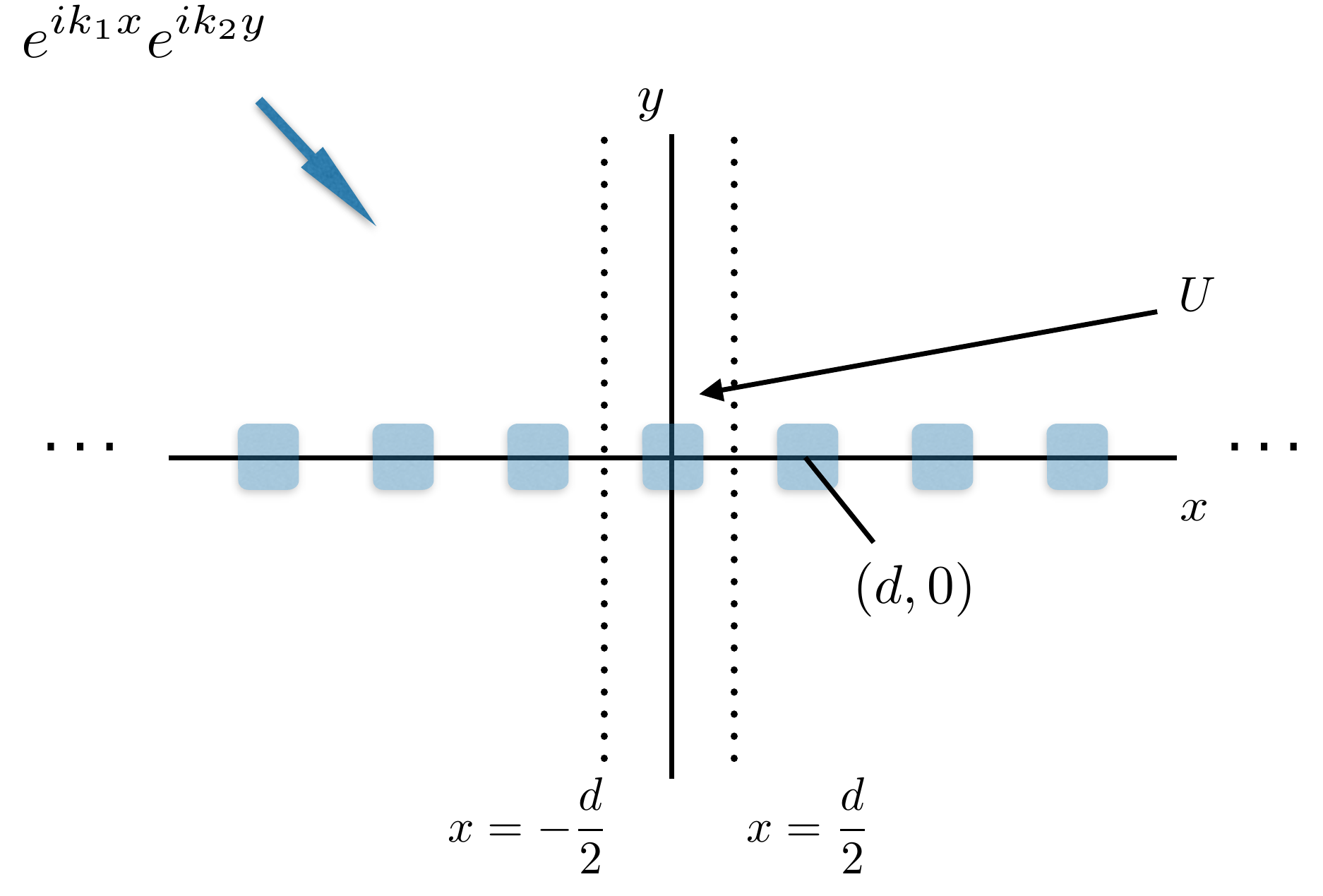}
\caption{On the left is a periodic two-dimensional array of 
spherical scatterers lying in the $xy$-plane, with an incident 
plane wave at frequency
$k = \sqrt{k_1^2 + k_2^2 + k_3^2}$ traveling in the direction 
$(k_1,k_2,k_3)$. The unit cell is denoted by $U$.
On the right is a periodic one-dimensional array of 
scatterers lying along the $x$-axis, with an incident 
plane wave at frequency
$k = \sqrt{k_1^2 + k_2^2}$ traveling in the direction 
$(k_1,k_2)$. The unit cell is again denoted by $U$.
\label{scatfig}}
\end{figure}

Without entering into the details of integral equation methods or
multiple scattering theory 
\cite{Chew,Morse,vandeHulst}, 
we note that the 
Green's function for the Helmholtz equation in free space 
is given by 
\[ G_{3D}(x,y,z) =  \frac{e^{i k r}}{4 \pi r} = 
{\left(\frac{i k}{4 \pi}\right)} h_0(kr),
\quad  G_{2D}(x,y,z) =  \frac{H_0(kr)}{4i} \, ,
\]
in three and two dimensions, respectively, where
$r = \sqrt{x^2+y^2 + z^2}$ or
$\sqrt{x^2+y^2}$ depending on the dimension, 
$h_0(x)$ denotes the spherical Hankel function of order zero, and
$H_0(x)$ denotes the standard Hankel function of order zero.
For quasiperiodic scattering, the Green's function can be expressed
formally as
\begin{align}
 G^{QP}_{3D}(x,y,z) &=  
\frac{i k}{4 \pi}  \sum_{n=-\infty}^\infty 
\sum_{m=-\infty}^\infty h_0(k\sqrt{(x-nd_1)^2 + (y-md_2)^2 + z^2})
e^{ink_1d_1} e^{imk_2d_2} \, ,
\label{qpg3d}
\\
G^{QP}_{2D}(x,y) &=  
\frac{1}{4 i}  \sum_{n=-\infty}^\infty 
H_0(k\sqrt{(x-nd)^2 + y^2})
e^{ik_1nd} \, .
\label{qpg2d}
\end{align}
Note that 
both formulas (\ref{qpg3d}) and (\ref{qpg2d}) are simply infinite series
of translated fundamental solutions.

By making use of standard addition theorems \cite{Morse}, it is 
straightforward and well-known that we can write
\[
 G^{QP}_{3D}(x,y,z) =  
\frac{i k}{4 \pi} 
\left[ h_0(kr) +  
\sum_{l=0}^\infty \sum_{j=-l}^l  S_{l,j} Y_l^{-j}(\theta,\phi) j_l(kr)
\right]  \, ,
\]
where 
\begin{align}
 S_{l,j} &= S_{l,j}(d_1,d_2,k,k_1,k_2) \nonumber \\
&= 
 \sum_{\substack{n,m = -\infty \\ (n,m) \neq (0,0)}}
^\infty 
 h_l(k\sqrt{(nd_1)^2 + (md_2)^2}) Y_l^j \left(\frac{\pi}{2},\phi_n^m \right)
e^{ink_1d_1} e^{imk_2d_2} \, .
\label{Snmdef}
\end{align}
Here, $h_n$ and $j_n$ denote the spherical Hankel and Bessel functions
of order $n$ and $\phi_n^m$ is the angle subtended by the point $(nd_1,md_2)$
with respect to the $x$-axis.
The function $Y_l^j$ denotes 
the spherical harmonic of degree $l$ and order $j$: 
\begin{equation}
Y_l^j(\theta,\phi) =
\sqrt{\frac{2l+1}{4 \pi}}
\sqrt{\frac{(l-|j|)!}{(l+|j|)!}}\, P_l^{|j|}(\cos \theta)
		      e^{i j \phi} \, ,
\label{ynmpnm}
\end{equation}
where the associated Legendre functions $P_l^j$ can be 
defined by the Rodrigues' formula
\[ P_l^j(x) = (-1)^j (1-x^2)^{j/2} \frac{d^j}{dx^j} P_l(x), \] 
with $P_l(x)$ the standard Legendre polynomial of degree $l$.

Similarly, in two dimensions, we have
\[
 G^{QP}_{2D}(x,y,z) =  
\frac{1}{4i} \left[  H_0(kr) +  
\sum_{l=-\infty}^\infty S_{l} J_l(kr) e^{i l \theta} \right] \, ,
\]
where 
$H_n$ and $J_n$ denote the usual Hankel and Bessel functions and
\begin{equation}
 S_{l} = S_{l}(d,k,k_1) = 
\sum_{n\in\mathbb{Z}\backslash\{0\}}
H_l(|n|kd) (\operatorname{sgn}n)^l e^{ik_1nd} \, .
\label{Sndef}
\end{equation}
The sums appearing in (\ref{Snmdef}) and (\ref{Sndef}) are referred to as 
{\em lattice sums} 
\cite{Linton2010,McPhedranGrubits2000,Moroz2001,Twersky1961}.
If the expressions in (\ref{qpg3d}), (\ref{qpg2d}) 
(\ref{Snmdef}) and (\ref{Sndef}) were well-defined, it would be 
straightforward to verify that 
$G^{QP}_{3D}$ and $G^{QP}_{2D}$ satisfy the desired quasiperiodicity
conditions. 

Unfortunately, three fundamental difficulties are encountered
in the use of lattice sums: they are conditionally convergent,
their ``direct" numerical evaluation
is extremely slow by naive methods, and they diverge for 
certain values of the wave parameters $k_1,k_2,k_3$ and lattice parameters
$d_1$ and $d_2$ - giving rise to what are 
known as Wood's anomalies \cite{Wood1902}.
The behavior of the scattered
field is quite striking in the neighborhood of those parameter values, 
as discussed in  \cite{BarnettGreengard2011,Linton2010,Maystre2012}
and in some detail below. In the two-dimensional case, 
it is straightforward to see from Fourier analysis
that the scattered field (away from the obstacle) must take the form
\cite{Linton2010}
\[ u^{scat}(x,y) = \sum_{n=-\infty}^\infty 
a_n
e^{2 \pi i nx/d} e^{i k_1 x} e^{i \beta_n y} \, ,
\]
for $y >0$ and 
\[ u^{scat}(x,y) = \sum_{n=-\infty}^\infty 
b_n
e^{2 \pi i nx/d} e^{i k \cos \psi x} e^{-i \beta_n y} \, ,
\]
for $y<0$,
where $\beta_n^2 + \left(k_1 + \frac{2\pi n}{d} \right)^2 = k^2$
with the root taken as positive real or positive imaginary. 
Real values of $\beta_n$ correspond to {\em propagating} modes, 
while values of $\beta_n$ on the positive imaginary axis 
correspond to {\em evanescent} modes. Wood's anomalies occur when 
$\beta_n = 0$ and the scattered wave is propagating exactly 
along the array in the $x$-direction - a very special type of physical 
resonance.
In three dimensions,
the scattered field (away from the plane of obstacles) must take the form
\[ u^{scat}(x,y,z) = \sum_{m,n \in \mathbb{Z}}
a_{m,n}
e^{2 \pi i nx/d_1} e^{i k_1 x} 
e^{2 \pi i my/d_2} e^{i k_2 y} 
e^{i \beta_{n,m} z} \, ,
\]
for $z >0$ and 
\[ u^{scat}(x,y,z) = \sum_{m,n \in \mathbb{Z}}
b_{m,n}
e^{2 \pi i nx/d_1} e^{i k_1 x} 
e^{2 \pi i my/d_2} e^{i k_2 y} 
e^{-i \beta_{n,m} z} \, ,
\]
for $z <0$, where
\begin{equation}
\beta_{n,m}^2 + \left(k_1 + \frac{2\pi n}{d_1} \right)^2 +
\left(k_2 + \frac{2\pi m}{d_2} \right)^2 = k^2.
\label{wood2d}
\end{equation}
Wood's anomalies occur when 
$\beta_{n,m} = 0$ and the scattered wave is propagating in some direction
along the $xy$-plane \cite{LiuBarnett2016,Shipman,LintonThompson2007}.

The computation of lattice sums and the resonant behavior corresponding
to Wood's anomalies have been widely studied, 
from both a physical and a mathematical perspective (see, for example,
\cite{BarnettGreengard2011,Kurkcu2009,Linton2010,LintonThompson2007,
LintonThompson2009,Maystre2012,McPhedranGrubits2000,Moroz2001,Moroz2006,
Shipman,Twersky1961,YasumotoYoshitomi1999}). 
Oddly enough, in the numerical literature
for evaluating lattice sums, the occurence of Wood's anomalies is 
often ignored, despite the fact that the series can diverge and
despite the potential loss 
of accuracy in computational results near such singularities.

In this paper, we focus on the 
rigorous analysis of one-dimensional 
lattice sums using a novel method based on quadrature, 
Euler-MacLaurin corrections to the trapezoidal rule, and steepest descent 
arguments. The reason for concentrating on the one-dimensional case is the 
remarkable work of 
McPhedran,  Nicorovici, Botten, Grubits, Enoch and Nixon
\cite{McPhedranNixon2001,McPhedranGrubits2000}, who showed that,
once the
one-dimensional lattice sums along the $x$-axis are obtained,
highlighted in blue on the left-hand side of Fig.  \ref{scatfig},
the remaining contributions
in higher dimensions can be computed semi-analytically using the Poisson
summation formula.
This technique, sometimes referred to as 
\emph{lattice reduction} \cite{LintonThompson2009}, 
is outlined briefly in section \ref{latticeredux}.

From a physical perspective, the correct choice of the conditionally
convergent lattice sums corresponds to adding an infinitesimal amount
of dissipation: that is, we replace 
the real wave number $k$ by $k+i \epsilon$ and then let
$\epsilon \rightarrow 0^+$. For any fixed $\epsilon > 0$,
the relevant infinite series converges absolutely.
Thus, instead of \eqref{Sndef} and \eqref{Snmdef},
we {\em define} the conditionally convergent
lattice sums by
\begin{equation}
 S_{l} = \lim_{\epsilon \rightarrow 0^+}
 \sum_{n\in\mathbb{Z}\backslash\{0\}}
H_l(|n|(k+i\epsilon) d) (\operatorname{sgn}n)^l e^{ik_1nd} \, 
\label{Sndefeps}
\end{equation}
and
\begin{equation}
 S_{l,j} = 
 \lim_{\epsilon \rightarrow 0^+}
 \sum_{\substack{n,m = -\infty \\ (n,m) \neq (0,0)}}
^\infty 
 h_l((k+i\epsilon)\sqrt{(nd_1)^2 + (md_2)^2}) Y_l^j 
\left(\frac{\pi}{2},\phi_n^m \right)
e^{ink_1d_1} e^{imk_2d_2} \, .
\label{Snmdefeps}
\end{equation}

The existence of these limits will
be established in our analysis, and the original
physical interpretation does not play a role. 
We will show that the one-dimensional sums may be evaluated through 
techniques of complex analysis, combined with the Euler-Maclaurin 
formula with superalgebraic convergence.
Our fast algorithm is derived in sections \ref{sec-1d} and \ref{secfast}, and
with proofs collected in sections \ref{sec-estimates1} and
\ref{sec-estimates2}. Section \ref{sec-numerical}
presents some numerical experiments, and section \ref{concl} contains
some concluding remarks.

\section{Lattice Reduction} \label{latticeredux}

We illustrate the lattice reduction technique of 
\cite{McPhedranNixon2001,McPhedranGrubits2000}
for the case shown on the left-hand side of Fig. \ref{scatfig},
with the added simplification that
we assume the lattice to be square with unit cell of area one 
($d_1 = d_2 = 1$). 
For the prescribed wavenumber $k$, the Bloch phases
are then given by $(e^{i\alpha}, e^{i\beta}) = (e^{ik_1}, e^{ik_2})$.
For the sake of brevity, we consider only the lattice sum
$S_{0,0}$ from (\ref{Snmdef}) which now takes the form:
\begin{align}
S_{0,0} &=  \sum_{(n,m)\neq (0,0)} h_0(k \sqrt{n^2 + m^2})
e^{i \alpha n} e^{i \beta m}  \nonumber \\
&= S_{0,0}^+ + S^{grating}_{0,0} + S_{0,0}^-  \, , \label{partial_sums} 
\end{align}
where
\[
S_{0,0,}^+ = \sum_{m>0} \sum_{n} h_0(k \sqrt{n^2 + m^2})
e^{i \alpha n} e^{i \beta m}  \, ,\quad
S_{0,0}^-  = \sum_{m<0} \sum_{n} h_0(k \sqrt{n^2 + m^2})
e^{i \alpha n} e^{i \beta m} \, ,
\]
and
\begin{equation}
S_{0,0}^{grating} = \sum_{n \neq 0}
 h_0(k |n|) e^{i \alpha n} \, . \label{puncsum}
\end{equation}
Let us consider the sum $S^{+}_{0,0}$, which we write in the form
\begin{equation*}
S^+_{0,0} = \sum_{m > 0}
\sum_{n\in\mathbb{Z}} 
h_0 ( k \sqrt{n^2 + m^2})
e^{i \alpha n} e^{i \beta m} \,  = \  \sum_{m > 0}
e^{i \beta m} s_m \, ,
\end{equation*}
where
\[
s_m = \sum_{n\in \mathbb{Z}} h_0 ( k \sqrt{n^2 + m^2} )
e^{i \alpha n} \, .
\]
($S^{-}_{0,0}$ is treated in an analogous fashion.) 

The important things to note about $s_m$ are 
(1) that it is the sum of the function $h_0(k \sqrt{x^2 + m^2}) e^{i \alpha x}$
sampled at the integers $n$, (2) that the function is smooth 
since $m>0$, and (3) that
the spherical Hankel function $h_0$ has the spectral representation
\cite{Chew,Morse,vandeHulst} 
\begin{equation*}
h_0(k\sqrt{x^2 + y^2 + z^2}) = \frac{1}{2 \pi i k}
\int_{-\infty}^{\infty} \int_{-\infty}^{\infty} 
\frac{e^{- y \sqrt{s^2 + t^2 - k^2}}}{\sqrt{s^2 + t^2 - k^2}}
e^{isx} e^{itz} \, ds\, dt \, ,
\end{equation*}
for $y>0$.
From this, letting $x =n$, $y=m$, and $z=0$,  we have
\begin{equation*}
s_m =
\frac{1}{2 \pi i k}
\sum_{n\in\mathbb{Z}} 
\int_{-\infty}^{\infty} \int_{-\infty}^{\infty} 
\frac{e^{- m \sqrt{s^2 + t^2 - k^2}}}{\sqrt{s^2 + t^2 - k^2}}
e^{in(s+\alpha)} \, ds\, dt\, .
\end{equation*}
We now apply the Poisson
summation formula 
\cite{McPhedranGrubits2000,Twersky1961},
which we write informally as
\[ \sum_n e^{inx} = 2 \pi \, \sum_n \delta(x+2\pi n). \]
This yields 
\begin{equation*}
s_m = 
\frac{1}{i k}
\sum_{n\in\mathbb{Z}} 
\int_{-\infty}^{\infty} 
\frac{e^{- m \sqrt{t^2-k^2 + (2 \pi n - \alpha)^2}}}
{\sqrt{t^2 - k^2 + (2 \pi n - \alpha)^2}} \, dt.
\end{equation*}
From this, we have
\begin{equation*}
\begin{aligned}
S^{+}_{0,0} = &
\frac{1}{i k}
\sum_{m > 0}\sum_{n\in\mathbb{Z}}
e^{i\beta m}
\int_{-\infty}^{\infty} 
\frac{e^{- m \sqrt{t^2-k^2 + (2 \pi n - \alpha)^2}}}
{\sqrt{t^2 - k^2 + (2 \pi n - \alpha)^2}} \, dt \\
=&
\frac{1}{ik}
\sum_{n\in\mathbb{Z}}
\int_{-\infty}^{\infty} 
\frac{e^{i\beta- \sqrt{t^2-k^2 + (2 \pi n - \alpha)^2}}}
{\sqrt{t^2 - k^2 + (2 \pi n - \alpha)^2}} \, 
\frac{1}
{1 - e^{i\beta - \sqrt{t^2 - k^2 + (2 \pi n - \alpha)^2}}} \, dt.
\end{aligned}
\end{equation*}
The last expression is obtained by summing a geometric series in the index
$m$. Assuming that this formal manipulation makes sense,
the resulting integral is rapidly converging in $t$ and the outer 
summation is rapidly converging in $n$. 
We omit further details, referring the 
reader to \cite{McPhedranNixon2001,McPhedranGrubits2000}. Suffice it to say
that, as a result of this observation, the 
principal obstacle in evaluating the lattice sum $S_{0,0}$ is that of
computing the one-dimensional  sum 
$S^{grating}_{0,0}$ in (\ref{puncsum}). 
This ``punctured sum" cannot be evaluated through the Poisson
summation formula directly, since the summand at $n=0$ is undefined 
(although Ewald type methods could be used to overcome 
this \cite{Linton2010}).

The remainder of this paper is devoted to a 
new approach for the punctured sum, which results in a fast algorithm and
may be of mathematical interest in its own right.

\section{One-dimensional lattice sums and the Euler- \\
MacLaurin formula} \label{sec-1d}

To develop a unified framework that can handle one-dimensional
sums such as 
$S^{grating}_{0,0}$ in (\ref{puncsum}) or $S_l$ in (\ref{Sndef}),
let us now fix $k \in \mathbb{R}$, $d>0$ and suppose 
that $f(z)$ is a complex-analytic
function such that, for some $\beta > 0$, there exists
a representation for the $p$th derivative of the form
\begin{equation}
f^{(p)}(z) = \frac{\Phi_p(z)}{z^\beta} e^{iz}\, ,
\end{equation}
where $\Phi_p(z)$ has an asymptotic series $\Phi_p(z)\sim\sum_{\mu=0}^{\infty}
c_{p,\mu} z^{-\mu}$ valid for $|\textnormal{arg} \, z|<\pi/2 + \delta$ and
$|z|\rightarrow\infty$, for sufficiently small $\delta > 0$. 
We also assume that $\Phi_p$ is bounded on the
region $\left\{ |z| > \delta^\prime, \; |\textnormal{arg} \,  z| < \pi/2
+ \delta\right\}$, for each $\delta^\prime > 0$. 
The Bessel functions $h_n$ and $H_n$ are well-known to satisfy such
estimates \cite{AbramowitzStegun1972}.

For $\epsilon>0$ and $\alpha\in\mathbb{R}$, we define the function
\begin{equation}
{F}_{\alpha}^{\epsilon}(z) = f\left((k+i\epsilon)z\right)
e^{i\alpha z} \, .
\end{equation}
We shall denote by $F_{\alpha}$ the
corresponding expression with $\epsilon=0$. For reasons that will become
apparent later, we shall assume that $\left\{
(k \pm \alpha)d/2\pi\right\} \cap \mathbb{Z} = \varnothing $.
We now consider absolutely convergent sums of the general form:
\begin{equation}
{S}^{\epsilon}(\alpha) = \sum_{n>0}
{F}_{\alpha}^{\epsilon}(nd) \, .
\end{equation}
Clearly, for $\ell \in \mathbb{Z}$, sums of the form
\[  \sum_{n \in \mathbb{Z} \setminus \{0\}} f( (k+i\epsilon)|nd|) 
(\textnormal{sgn}\, n)^l e^{i\alpha nd} = {S}^{\epsilon}(\alpha) +
(-1)^\ell {S}^{\epsilon}(-\alpha), \]
so we shall concern ourselves
primarily with ${S}^{\epsilon}(\alpha)$. 
This framework covers the cases of physical interest in quasiperiodic 
scattering.

Our approach 
will require the use of an infinitely differentiable filter function,
which we introduce here.

\begin{definition} \label{filtdef}
Let $\psi\in C^{\infty}([0,1])$  be a decreasing function
such that, for some $r$ with $0 <r < \frac{1}{4}$, 
$\psi|_{[0,2r]}=1$ and $\psi|_{[1-2r,1]}=0$.
For $b,c\in\mathbb{N}$, we denote
the scaled filter function 
$\psi_{b,c} (x) $ by
\begin{equation}
\psi_{b,c}(x) = \begin{cases}
1,& x<b \, ,\\
\psi\left[\frac{1}{c}(x-b)\right],& b\leq x\leq b+c \, ,\\
0,& b+c<x \, .
\end{cases}
\end{equation}
See Fig. \ref{figfilt}. Note that the parameter $r$ determines both how flat the
scaled filter function $\psi_{b,c}(x)$ is at the points $x=b$ and $x=b+c$, and 
how steep the transition is from $1$ to $0$.
\end{definition}

From the properties of $\psi_{b,c}$, we may write 
\begin{equation*}
{S}^{\epsilon}(\alpha)=\sum_{n=1}^{b-1}{F}_{\alpha}^{\epsilon}(nd)
+ \sum_{n=b}^{b+c} {F}_{\alpha}^{\epsilon}(nd)\psi_{b,c}(n) +
\sum_{n=b}^{\infty}{F}_{\alpha}^{\epsilon}(nd)(1-\psi_{b,c}(n)) \, .
\end{equation*}

Note that the third term corresponds to the trapezoidal approximation for 
the integral
\[
\int_{b}^{\infty}{F}_{\alpha}(xd)(1-\psi_{b,c}(x)) \, dx.
\] 
We now show that this integral can be evaluated 
by contour deformation and that the difference between the integral and the 
desired sum can be computed with high precision using Euler-MacLaurin corrections.

\begin{lemma} \label{contourlemma}
Let $\phi\in(-\pi/2,\pi/2)$, such that $\tan\phi=(k+\alpha)/\epsilon$.
Then
\[
\int_{b}^{\infty} {F}_{\alpha}^{\epsilon}(xd)\, dx
= \, \frac{1}{d} \int_{0}^\infty
 {F}_{\alpha}^{\epsilon}(bd+\zeta\tau)\, \zeta \, d\tau
\]
where 
$\zeta = e^{i \phi} = \frac{\epsilon+i(k+\alpha)}{\sqrt{\epsilon^2+(k+\alpha)^2}}$.
\end{lemma}

\begin{proof}

Let $C_0$, $C_\phi$ denote the outward rays
in the complex plane with $\arg z = 0$ and $\arg z = \phi$. 
We begin by writing
\begin{equation*}
\int_{b}^{\infty} {F}_{\alpha}^{\epsilon}(xd)\, dx
=\frac{1}{d} \int_{0}^{\infty}{F}_{\alpha}^{\epsilon}(bd+t)\, dt
=\frac{1}{d} \int_{C_0} {F}_{\alpha}^{\epsilon}(bd+z)\, dz
\end{equation*}
Now let $C_0(R) = [0,R]$ and let $C_\phi(R) = [0,R\zeta]$.
Let $B(R)$ denote the circular arc from $R$ to $R\zeta$, let 
$I=[0,\phi]$ if $\phi > 0$, and let
$I = [\phi,0]$ if $\phi < 0$.
It is straightforward to verify that
\begin{equation*}
\begin{aligned}
\left|\int_{B(R)}{F}_{\alpha}^{\epsilon}(bd+z)\,dz\right|
=&\left|\int_I
{F}_{\alpha}^{\epsilon}(bd+R e^{i\phi^\prime})
R i e^{i\phi^\prime}
d\phi^\prime\right|\,
\leq \frac{C}{(|k|bd)^\beta}\cdot R e^{-\epsilon R \cos\phi} \, ,
\end{aligned}
\end{equation*}
which goes to zero as $R\rightarrow\infty$.
This justifies the contour deformation from $C_0$ to $C_\phi$:
\begin{equation*}
\int_{b}^{\infty} {F}_{\alpha}^{\epsilon}(xd)\, dx
= \, \frac{1}{d} \int_{C_\phi} {F}_{\alpha}^{\epsilon}(bd+z)\, dz
= \, \frac{1}{d} \int_{0}^{\infty} {F}_{\alpha}^{\epsilon}
\left(bd+\zeta \tau\right)
\zeta \, d\tau \, .
\end{equation*}
\hfill $\square$
\end{proof}

\begin{figure}[h]
\centering
\includegraphics[width=0.6\textwidth]{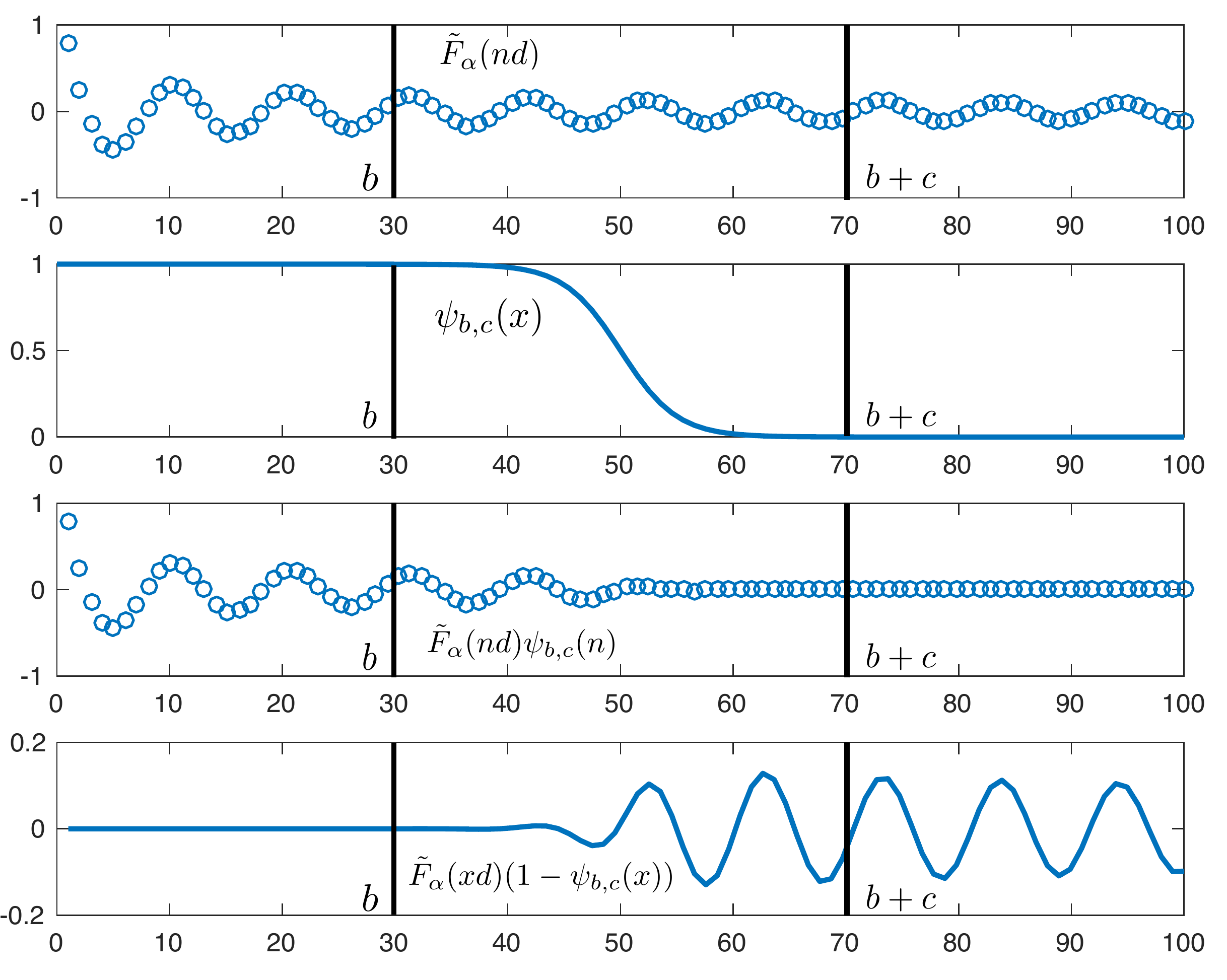}
\caption{{\em Lattice sums and the Euler-MacLaurin formula}:
The top row shows the discrete values of ${F}_\alpha(nd)$ whose sum
corresponds to ${S}(\alpha)$. The second row depicts an infinitely
differentiable filter function
$\psi_{b,c}(x)$, which equals $1$ for $x\leq b$ and $0$ for
$x>b+c$.  The third row is a plot of the discrete values
${F}_{\alpha}(nd)\psi_{b,c}(n)$ and the bottom row is a plot of 
${F}_{\alpha}(xd)(1-\psi_{b,c}(x))$. Our method exploits the relation
between the desired lattice sum and the improper integral
$\int_{n=b}^{\infty}{F}_{\alpha}(xd)(1-\psi_{b,c}(x)) \, dx,$ 
which involves Euler-MacLaurin corrections at the endpoints.}
\label{figfilt}
\end{figure}

\begin{theorem}
Let $\psi_{b,c}$ be a scaled filter function (Definition \ref{filtdef}.)
Then, for any fixed integer $s\geq 2$, we have
\begin{equation}
\label{eq:algorithm_formula}
\begin{aligned}
{S}^{\epsilon}(\alpha)=
&\sum_{n=1}^{b-1}{F}_{\alpha}^{\epsilon}(nd)
+\sum_{n=b}^{b+c}{F}_{\alpha}^{\epsilon}(nd)\psi_{b,c}(n)
-\int_{b}^{b+c}{F}_{\alpha}^{\epsilon}(xd)\psi_{b,c}(x)\,dx\\
&+\frac{1}{d} \int_{0}^{\infty} {F}_{\alpha}^{\epsilon}
\left(bd+\frac{\epsilon+i(k+\alpha)}{\sqrt{\epsilon^2+(k+\alpha)^2}}\tau\right)
\frac{\epsilon+i(k+\alpha)}{\sqrt{\epsilon^2+(k+\alpha)^2}}\,d\tau
+R_{s,b,c}^{\epsilon}(\alpha) \, ,
\end{aligned}
\end{equation}
where 
\begin{equation}
\label{eq:err_formula}
R_{s,b,c}^{\epsilon}(\alpha) = \frac{(-1)^{s-1}}{s!}
\int_{b}^{\infty} B_s(x-[x])\left[\left.\frac{d^s}{dt^s}\left(
{F}_{\alpha}^{\epsilon}(td)(1-\psi_{b,c}(t))\right)\right|_{t=x}\right]\,dx \, ,
\end{equation}
and $B_s(x)$ denotes the Bernoulli polynomial of order $s$.
\end{theorem}

\begin{proof}

We have
\begin{equation}
{S}^{\epsilon}(\alpha)= \sum_{n>0}{F}_{\alpha}^{\epsilon}(nd)\\
= \, \sum_{n>0}{F}_{\alpha}^{\epsilon}(nd)\psi_{b,c}(n) +
\sum_{n>0}{F}_{\alpha}^{\epsilon}(nd)(1-\psi_{b,c}(n)) \, .
\end{equation}
Since $\psi_{b,c}(n)=1$, for $n\leq b$, and $\psi_{b,c}(n)=0$, for
$n>b+c$, we have
\begin{equation*}
{S}^{\epsilon}(\alpha)=\sum_{n=1}^{b-1}{F}_{\alpha}^{\epsilon}(nd)
+ \sum_{n=b}^{b+c} {F}_{\alpha}^{\epsilon}(nd)\psi_{b,c}(n) +
\sum_{n=b+1}^{\infty}{F}_{\alpha}^{\epsilon}(nd)(1-\psi_{b,c}(n)) \, .
\end{equation*}
We now employ the Euler-Maclaurin formula
(see \cite{AndrewsRoy1999}, p. 619, Theorem D.2.1) to find that, for
$s\in\mathbb{N}\backslash\{0\}$, $N>b$,
\begin{equation*}
\begin{aligned}
\sum_{n=b+1}^{N} {F}_{\alpha}^{\epsilon}(nd)(1-\psi_{b,c}(n)) =
&\int_{b}^{N} {F}_{\alpha}^{\epsilon}(xd)(1-\psi_{b,c}(x)) \,dx \\
&+\sum_{\gamma=1}^{s} \frac{(-1)^{\gamma}B_{\gamma}}{\gamma!}
\left.\frac{d^{\gamma-1}}{dx^{\gamma-1}}\left({F}_{\alpha}^{\epsilon}(xd)
(1-\psi_{b,c}(x))\right)\right|_{x=N} \\ 
&-\sum_{\gamma=1}^{s} \frac{(-1)^{\gamma}B_{\gamma}}{\gamma!} \left.
\frac{d^{\gamma-1}}{dx^{\gamma-1}}
\left({F}_{\alpha}^{\epsilon}(xd)(1-\psi_{b,c}(x))\right)\right|_{x=b} \\
&+\frac{(-1)^{s-1}}{s!} \int_{b}^{N} B_{s}(x-[x])\left[\left.\frac{d^s}{dt^s}
\left({F}_{\alpha}^{\epsilon}(td)(1-\psi_{b,c}(t))\right)
\right|_{t=x}\right]\, dx\, ,
\end{aligned}
\end{equation*}
where the $B_{s}(x)$ are the Bernoulli polynomials, and the $B_{\gamma}$ are
Bernoulli numbers. Observe that $(1-\psi_{b,c}(x))$ is identically zero in a
neighborhood of $x=b$ and that the boundary term at $x=b$ vanishes outright.
Moreover, ${F}_{\alpha}^{\epsilon}(xd)$ and all its derivatives tend
to zero exponentially as $x\rightarrow\infty$. Thus, only the two integral
terms on the right hand side of the Euler-Maclaurin formula are preserved in
the limit $N\rightarrow\infty$. 

Thus, 
\begin{equation*}
\begin{aligned}
\sum_{n=b+1}^{\infty} {F}_{\alpha}^{\epsilon}(nd)(1-\psi_{b,c}(n))=
&\int_{b}^{\infty}{F}_{\alpha}^{\epsilon}(xd)(1-\psi_{b,c}(x))\, dx+
R_{s,b,c}^{\epsilon}(\alpha) \\
=&\int_{b}^{\infty}{F}_{\alpha}^{\epsilon}(xd)\, dx-
\int_{b}^{b+c}{F}_{\alpha}^{\epsilon}(xd)\psi_{b,c}(x)\, dx+
R_{s,b,c}^{\epsilon}(\alpha)  \, ,
\end{aligned}
\end{equation*}
and
{\small 
\begin{equation*}
{S}^{\epsilon}(\alpha)=\sum_{n=1}^{b-1}{F}_{\alpha}^{\epsilon}(nd)
+\sum_{n=b}^{b+c}{F}_{\alpha}^{\epsilon}(nd)\psi_{b,c}(n)
-\int_{b}^{b+c}{F}_{\alpha}^{\epsilon}(xd)\psi_{b,c}(x)\, dx
+\int_{b}^{\infty}{F}_{\alpha}^{\epsilon}(xd)\, dx
+R_{s,b,c}^{\epsilon}(\alpha)  \, .
\end{equation*}
}
The desired result now follows from Lemma \ref{contourlemma}.
\hfill $\square$
\end{proof}

Before turning to the algorithm itself, we will require two more results,
whose proofs are deferred to the next sections.

\begin{lemma} \label{convlemma}
The term $R_{s,b,c}^{\epsilon}(\alpha)$ from eq. \eqref{eq:err_formula}
satisfies the estimate
\[
\lim_{c\rightarrow\infty}
\underset{\epsilon\rightarrow 0^+}{\overline{\lim}}
\left|R_{s,b,c}^{\epsilon}(\alpha) \right|= \mathcal{O}(c^{-N-\beta})
\]
for any fixed $N$.
\end{lemma}

\begin{lemma} \label{mixlemma}
Recall that ${F}_{\alpha}^{\epsilon}(z) = f((k+i\epsilon)z)e^{i \alpha z}$
and for $p\in \mathbb{N}\cup \left\{ 0 \right\}$ and
$q\in\mathbb{Z}$ let
\begin{equation}
\label{Galphapqdef}
G_{\alpha,p,q}^{\epsilon}(x)=
- \int_{x}^{\infty} f^{(p)}((k+i\epsilon)td) e^{i\alpha td}
e^{2\pi i q t} \,dt \, .
\end{equation}
Then $G_{\alpha,p,q}(x) \equiv
\lim_{\epsilon \rightarrow 0^+} G_{\alpha,p,q}^\epsilon (x)$ exists 
(see formula (\ref{eq:lim_G_eps})) and
\begin{equation}
\begin{aligned}
\lim_{c\rightarrow\infty}
\limsup_{\epsilon\rightarrow 0^+} &\left|
\sum_{n=b}^{b+c}{F}_{\alpha}^{\epsilon}(nd)\psi_{b,c}(n)
-\int_{b}^{b+c}{F}_{\alpha}^{\epsilon}(xd)\psi_{b,c}(x)dx
- \mathcal{A} \right|=0
\end{aligned}
\end{equation} 
where
\begin{equation}
\begin{aligned}
& \mathcal{A} =
{F}_{\alpha} (bd)
-\sum_{\gamma=1}^{s} \frac{(-1)^{\gamma}B_{\gamma}}{\gamma!} \left.
\frac{d^{\gamma-1}}{dx^{\gamma-1}}
\left({F}_{\alpha} (xd)\right)\right|_{x=b} \\ 
& - \frac{(-1)^{s-1}}{s!} A_s \sum_{j_1=0}^s \binom{s}{j_1} d^s
k^{s-j_1} (i\alpha)^{j_1} \left[ \sum_{q \in \mathbb{Z}\backslash
\left\{0\right\}} q^{-s} G_{\alpha,s-j_1,q} (b) \right] \, ,
\end{aligned}
\end{equation}
where 
$A_{s}=-s!/(2\pi i)^{s}$ (see \cite{AbramowitzStegun1972}, 23.1.16).
\end{lemma}

\section{Proof of Lemma \ref{convlemma}} \label{sec-estimates1}

For Lemma \ref{convlemma},
we need to estimate the remainder
$R_{s,b,c} ^\epsilon(\alpha)$, defined via (\ref{eq:err_formula}).
This is precisely the error incurred by using the
Euler-Maclaurin formula.

We begin by using the generalized product rule
to expand the derivatives in (\ref{eq:err_formula}):
\begin{equation*}
\begin{aligned}
R_{s,b,c}^{\epsilon}(\alpha)=
&\frac{(-1)^{s-1}}{s!}
\int_{b}^{\infty} B_s(x-[x])\left[\left.\frac{d^s}{dt^s}\left(
{F}_{\alpha}^{\epsilon}(td)(1-\psi_{b,c}(t))
\right)\right|_{t=x}\right]\,dx\\
=&\frac{(-1)^{s-1}}{s!} \sum_{j=0}^{s} \sum_{j_1 =0}^{s-j} \binom{s}{j}
\binom{s-j}{j_1} d^{s-j} (k+i\epsilon)^{s-j-j_1} (i\alpha)^{j_1}\\
&\times\int_{b}^{\infty} B_{s}(x-[x]) (\delta_{j,0}-\psi_{b,c}^{(j)}(x))
f^{(s-j-j_1)}((k+i\epsilon)xd)e^{i\alpha xd} \,dx \, .
\end{aligned}
\end{equation*}
We see that it suffices to verify the smallness of finitely many terms of
the following form:
\begin{equation}
\begin{aligned}
\mathcal{R}_{\alpha,b,c,p,j}^\epsilon=\int_{b}^{\infty}
B_{s}(x-[x])(\delta_{j,0}-\psi_{b,c}^{(j)}(x))f^{(p)}((k+i\epsilon)xd)
e^{i\alpha xd}\, dx \, .
\end{aligned}
\end{equation}
The functions $B_{s}(x-[x])$ are Bernoulli polynomials, and for
$s \geq 2$ we have the absolutely convergent Fourier expansion:
\begin{equation*}
B_{s}(x-[x])=A_{s}\cdot\sum_{q\in\mathbb{Z}\backslash\{0\} }
q^{-s} e^{2\pi i q x} \, ,
\end{equation*}
with $A_{s}=(-1)s!/(2\pi i)^{s}$ (see \cite{AbramowitzStegun1972}, 23.1.16).
We shall assume henceforth that $s\geq 2$ to guarantee convergence; then by Fubini's theorem, we have
\begin{equation*}
\mathcal{R}_{\alpha,b,c,p,j}^{\epsilon} =A_{s}
\sum_{q \in\mathbb{Z}\backslash\{0\}}
q^{-s} \int_{b}^{\infty} (\delta_{j,0}-\psi_{b,c}^{(j)}(x))
f^{(p)}((k+i\epsilon)xd) e^{i\alpha xd} e^{2\pi i q x} \,dx \, .
\end{equation*}
Let us define (for $x>0$)
\begin{equation}
G_{\alpha,p,q}^{\epsilon}(x)=
- \int_{x}^{\infty} f^{(p)}((k+i\epsilon)td) e^{i\alpha td}
e^{2\pi i q t} \,dt \, .
\end{equation}
Integrating by parts, we have
\begin{equation*}
\begin{aligned}
\mathcal{R}_{\alpha,b,c,p,j}^{\epsilon} =
&A_{s} \sum_{q\in\mathbb{Z}\backslash\{0\}} q^{-s}
\int_{b}^{\infty} \psi_{b,c}^{(j+1)}(x) G_{\alpha,p,q}^{\epsilon}(x) \,dx \\
=&A_{s} \sum_{q\in\mathbb{Z}\backslash\{0\}} q^{-s}
\int_{b}^{b+c} \psi_{b,c}^{(j+1)}(x) G_{\alpha,p,q}^{\epsilon}(x) \,dx \\
=&A_{s} c^{-j} \sum_{q\in\mathbb{Z}\backslash\{0\}} q^{-s}
\int_{0}^{1} \psi^{(j+1)}(x) G_{\alpha,p,q}^{\epsilon}(b+cx) \,dx \, .
\end{aligned}
\end{equation*}

Let $\lambda_q = (\alpha+k)d+2\pi q$, and let
$\psi\in(-\pi/2,\pi/2)$ such that $\tan\psi=\lambda_q ( d \epsilon )^{-1}$.
Since we assumed that $(k \pm \alpha)d/2\pi\notin\mathbb{Z}$, we have that
$\lambda_q\neq 0$ for any $q\in\mathbb{Z}$; indeed, it holds that
$\inf_{q\in\mathbb{Z}} |\lambda_q| > 0$.
Let $C_{0}$, $C_{\psi}$ be contours in the $z$-plane along the outward rays
$\arg z=0$ and $\arg z=\psi$. We have
\begin{equation*}
\begin{aligned}
G_{\alpha,p,q}^{\epsilon}(x)=
&-\int_{x}^{\infty} f^{(p)}((k+i\epsilon)td) e^{i\alpha td}
e^{2\pi i q t} \,dt\\
=& -\int_{C_{0}}\frac{\Phi_p((k+i\epsilon)(x+z)d)}
{[(k+i\epsilon)(x+z)d]^{\beta}}
e^{i(\lambda_q +i d \epsilon)(x+z)} \,dz\\
=& -e^{i(\lambda_q +i d \epsilon )x}
\int_{C_{0}} \frac{\Phi_p((k+i\epsilon)(x+z)d)}{[(k+i\epsilon)(x+z)d]^{\beta}}
e^{i(\lambda_q + i d \epsilon )z} \,dz\\
\end{aligned}
\end{equation*}
We now deform the contour $C_0$ into $C_\psi$, which will be justified momentarily.
\begin{equation*}
\begin{aligned}
G_{\alpha,p,q}^\epsilon (x) 
=& -e^{i(\lambda_q +i d \epsilon )x}
\int_{C_{\psi}} \frac{\Phi_p((k+i\epsilon)(x+z)d)}
{[(k+i\epsilon)(x+z)d]^{\beta}}
e^{i(\lambda_q +i d \epsilon )z} \,dz\\
=&-e^{i(\lambda_q +i d \epsilon)x}
\int_{0}^{\infty}\frac{\Phi_p((k+i\epsilon)(x+e^{i\psi}\tau)d)}
{[(k+i\epsilon)(x+e^{i\psi}\tau)d]^{\beta}}
e^{i(\lambda_q +i d\epsilon)(\cos\psi+i\sin\psi)\tau} e^{i\psi} \,d\tau\\
=&-e^{i(\lambda_q +i d \epsilon )x}
\int_{0}^{\infty}\frac{\Phi_p((k+i\epsilon)(x+e^{i\psi}\tau)d)}
{[(k+i\epsilon)(x+e^{i\psi}\tau)d]^{\beta}}
e^{-\tau\sqrt{\lambda_q^2+d^2 \epsilon^2}} e^{i\psi} \,d\tau \, .
\end{aligned}
\end{equation*}
This implies that $G_{\alpha,p,q}^{\epsilon}(x)$ is bounded on  $[1,\infty)$ 
uniformly as $\epsilon\rightarrow 0^{+}$, $|q|\rightarrow\infty$, 
moreover, by the dominated convergence theorem
(with $\eta=\operatorname{sgn}\lambda_q$),
\begin{equation}
\label{eq:lim_G_eps}
\lim_{\epsilon\rightarrow 0^+} G_{\alpha,p,q}^{\epsilon}(x)
=-e^{i\lambda_q x} i \eta
\int_{0}^{\infty}\frac{\Phi_p((x+i\eta\tau)kd)}{[(x+i\eta\tau)kd]^{\beta}}
e^{-|\lambda_q | \tau} \, d\tau \, .
\end{equation}

The contour deformation in the above calculation, which takes place at
some \emph{fixed} $\epsilon > 0$, is justified by the following
estimate, where $B(R)$ and $I$ are defined as on 
the proof of Lemma \ref{contourlemma}.
\begin{equation*}
\begin{aligned}
& \  \left|\int_{B(R)}
 \frac{\Phi_p((k+i\epsilon)(x+z)d)}{[(k+i\epsilon)(x+z)d]^{\beta}} 
e^{i(\lambda_q +i d \epsilon)z} dz\right| \hspace{1.3in} \\
=& \ \left|\int_I
\frac{\Phi_p((k+i\epsilon)(x+R e^{i\psi^\prime })d)}
{[(k+i\epsilon)(x+R e^{i\psi^\prime} )d]^{\beta}}
e^{i R (\lambda_q +i d\epsilon )
(\cos\psi^\prime+i \sin\psi^\prime)}
R i e^{i\psi^\prime }
d\psi^\prime \right|\\
\leq& \  C R (|k|xd)^{-\beta} \int_I
e^{-|\lambda_q | R \sin |\psi^\prime|}
e^{-\epsilon d R \cos\psi^\prime}
d\psi^\prime  \\
\leq& \  C (|k|xd)^{-\beta} \cdot R e^{-\epsilon d R \cos\psi}
\longrightarrow 0 \textnormal{ as } R\rightarrow\infty \, .
\end{aligned}
\end{equation*}

Let us define $G_{\alpha,p,q}(x)=\lim_{\epsilon\rightarrow 0^+}
G_{\alpha,p,q}^{\epsilon}(x)$; we desire an asymptotic approximation
to this function as $x\rightarrow\infty$. Recall that we have an asymptotic
approximation 
$\Phi_p (z)\sim\sum_{\mu=0}^{\infty}c_{p,\mu} z^{-\mu}$,
$|z|\rightarrow\infty$. Thus, for any \emph{fixed} $N\geq 1$, we have
\begin{equation*}
\begin{aligned}
G_{\alpha,p,q}(x)
=&-e^{i\lambda_q x} i \eta
\int_{0}^{\infty}\frac{\Phi_p((x+i\eta\tau)kd)}{[(x+i\eta\tau)kd]^{\beta}}
e^{-|\lambda_q | \tau} \,d\tau\\
=&-e^{i\lambda_q x} i \eta
\left(\sum_{\mu=0}^{N}\frac{c_{p,\mu}}{(kd)^{\mu+\beta}}
\int_{0}^{\infty}
\frac{e^{-|\lambda_q | \tau}}{(x+i\eta\tau)^{\mu+\beta}} \,d\tau
+ \mathcal{O}(x^{-N-\beta})\right) \, ,\\ 
\end{aligned}
\end{equation*}
Now we integrate by parts iteratively, finitely many times, until the error
term is $\mathcal{O} (x^{-N-\beta})$.
\begin{equation*}
\begin{aligned}
G_{\alpha,p,q} (x) =&
-e^{i\lambda_q x} i \eta
\left(\sum_{\mu=0}^{N}\sum_{\nu=0}^{N-\mu}
\frac{c_{p,\mu}(i\eta)^{\nu} \prod_{0\leq i_1 <\nu}(-\mu-\beta-i_1)}
{(kd)^{\mu+\beta}|\lambda_q |^{\nu+1}}
x^{-(\mu+\nu)-\beta}
+ \mathcal{O}(x^{-N-\beta})\right)\\
=&-e^{i\lambda_q x} i \eta
\left(\sum_{\rho=0}^{N}\sum_{\mu=0}^{\rho}
\frac{c_{p,\mu}(i\eta)^{\rho-\mu} \prod_{0\leq i_1 <(\rho-\mu)}
(-\mu-\beta-i_1)}
{(kd)^{\mu+\beta}|\lambda_q |^{\rho-\mu+1}}
x^{-\rho-\beta}
+ \mathcal{O}(x^{-N-\beta})\right)\\
=&-e^{i\lambda_q x} i \eta
\left(\sum_{\rho=0}^{N-1}
\left[\sum_{\mu=0}^{\rho}
\frac{c_{p,\mu}(i\eta)^{\rho-\mu} \prod_{0\leq i_1<(\rho-\mu)}
(-\mu-\beta-i_1)}
{(kd)^{\mu+\beta}|\lambda_q |^{\rho-\mu+1}}\right]
x^{-\rho-\beta}
+ \mathcal{O}(x^{-N-\beta})\right) \, .
\end{aligned}
\end{equation*}
In particular, we find that
$G_{\alpha,p,q}(x)\sim e^{i\lambda_q x}
\sum_{\rho=0}^{\infty} G_{\rho,p,q} x^{-\rho-\beta}$; moreover, since
$k$ and $d$ are fixed, we have an estimate
$|G_{\rho,p,q}|\leq C_{\rho,p}\cdot(1+|\lambda_q |^{-\rho-1})$.
Since $\left\{(k \pm \alpha)d/2\pi\right\}\cap\mathbb{Z}=\varnothing$, we have that
$\inf_q |\lambda_q|>0$, so for fixed $\rho$,
$G_{\rho,p,q}$ is uniformly bounded in $q$. Additionally, the
error estimate $\mathcal{O}(x^{-N-\beta})$ is uniform in $q$, in that
both the constant and the domain of validity are uniform in $q$, if
$N$ is held constant.
(This relies on $\inf_q |\lambda_q|>0$.)

Recall that
\begin{equation*}
\mathcal{R}_{\alpha,b,c,p,j}^{\epsilon}
=A_{s} c^{-j} \sum_{q \in\mathbb{Z}\backslash\{0\}} q^{-s}
\int_{0}^{1} \psi^{(j+1)}(x) G_{\alpha,p,q}^{\epsilon}(b+cx) \,dx \, .
\end{equation*}
Since $G_{\alpha,p,q}^{\epsilon}(x)$ is uniformly bounded in the
simultaneous limit $\epsilon\rightarrow 0^+$, $|q|\rightarrow\infty$,
we may use the dominated convergence theorem to commute the limit
$\epsilon\rightarrow 0^+$ with the integral, as well as the sum
over $q$:
\begin{equation*}
\lim_{\epsilon\rightarrow 0^+}\mathcal{R}_{\alpha,b,c,p,j}^{\epsilon}
=A_{s} c^{-j} \sum_{q\in\mathbb{Z}\backslash\{0\}} q^{-s}
\int_{0}^{1} \psi^{(j+1)}(x) G_{\alpha,p,q}(b+cx) \, dx\, .
\end{equation*}
By using
$G_{\alpha,p,q}(x) = e^{i\lambda_q x}
\sum_{\rho=0}^{N-1} G_{\rho,p,q} x^{-\rho-\beta} + 
\mathcal{O}(x^{-N-\beta})$ 
(the error estimate being uniform in $q$),
and the fact that $\forall j\geq 0$, $\psi^{(j+1)}(x)$ is 
\emph{identically zero} on $[0,2r]$, we have (in the limit $c\rightarrow\infty$)
\begin{equation*}
\begin{aligned}
&\lim_{\epsilon\rightarrow 0^+}\mathcal{R}_{\alpha,b,c,p,j}^{\epsilon}\\
=&A_{s} c^{-j}
\sum_{\rho=0}^{N-1}
\sum_{q \in\mathbb{Z}\backslash\{0\}}
\frac{G_{\rho,p,q}e^{i\lambda_q b}}{q^s}
\int_{r}^{1-r} \psi^{(j+1)}(x)
\frac{e^{i\lambda_q cx}}{(b+cx)^{\rho+\beta}} \,dx
+ \mathcal{O}(c^{-j-N-\beta})\\
=&A_{s} c^{-j}
\sum_{\rho=0}^{N-1}
\sum_{q \in\mathbb{Z}\backslash\{0\}}
\frac{G_{\rho,p,q}e^{i\lambda_q b}}{c^{\rho+\beta} q^{s}}
\int_{r}^{1-r} \frac{\psi^{(j+1)}(x)}{x^{\rho+\beta}}
e^{i\lambda_q cx}\left(1+\frac{b}{cx}\right)^{-\rho-\beta} \,dx
+ \mathcal{O}(c^{-j-N-\beta})\\
\end{aligned}
\end{equation*}
Since $c$ is large and $x \geq r$, we may expand
$(1+b/cx)^{-\rho-\beta}$ in powers of $(b/cx)$.
\begin{equation*}
\begin{aligned}
&\lim_{\epsilon\rightarrow 0^+}\mathcal{R}_{\alpha,b,c,p,j}^{\epsilon} \\
=&A_{s}
\sum_{\rho=0}^{N-1} \sum_{\omega=0}^{N-\rho-1}\left[
(-1)^\omega
\frac{\prod_{i_1=0}^{\omega-1}(\rho+i_1+\beta)}
{\omega! b^{-\omega} c^{j+\rho+\omega+\beta}}
\sum_{q \in\mathbb{Z}\backslash\{0\}}
\frac{G_{\rho,p,q }e^{i\lambda_q b}}{q^{s}}
\int_{r}^{1-r} \frac{\psi^{(j+1)}(x)}{x^{\rho+\omega+\beta}}
e^{i\lambda_q cx} \,dx\right]\\
&+ \mathcal{O}(c^{-j-N-\beta}) \, .
\end{aligned}
\end{equation*}
Note that the integral on the right hand side depends on $c$ only through
$e^{i\lambda_q cx}$, and $\psi^{(j+1)}(x)/x^{\rho+\omega+\beta}$ is
smooth and compactly supported on $[r,1-r]$. Thus, the integral decays faster
than any power of $c^{-1}$ as $c\rightarrow\infty$, and such decay is uniform
in $q$. 

Thus, as $c\rightarrow\infty$, for any fixed $N$, we have
\[ \lim_{\epsilon\rightarrow 0^+}\mathcal{R}_{\alpha,b,c,p,j}^{\epsilon} 
= \mathcal{O}(c^{-j-N-\beta}),
\]
yielding the desired result.
\hfill $\square$

\section{Proof of Lemma \ref{mixlemma}} \label{sec-estimates2}

For Lemma \ref{mixlemma} we employ the 
Euler-Maclaurin formula yet again.

\begin{equation*}
\begin{aligned}
&\sum_{n=b}^{b+c} {F}_{\alpha}^{\epsilon}(nd)\psi_{b,c}(n) -
\int_{b}^{b+c} {F}_{\alpha}^{\epsilon}(xd)\psi_{b,c}(x) \,dx \\
=&{F}_{\alpha}^{\epsilon}(bd)
+\sum_{\gamma=1}^{s} \frac{(-1)^{\gamma}B_{\gamma}}{\gamma!}\left[
\left.\frac{d^{\gamma-1}}{dx^{\gamma-1}}\left({F}_{\alpha}^{\epsilon}(xd)
\psi_{b,c}(x)\right)\right|_{x=b+c}
- \left.
\frac{d^{\gamma-1}}{dx^{\gamma-1}}
\left({F}_{\alpha}^{\epsilon}(xd)\psi_{b,c}(x)\right)\right|_{x=b}
\right]\\
&+\frac{(-1)^{s-1}}{s!} \int_{b}^{b+c} B_{s}(x-[x])\left[\left.\frac{d^s}{dt^s}
\left({F}_{\alpha}^{\epsilon}(td)\psi_{b,c}(t)\right)
\right|_{t=x}\right]\,dx \, .
\end{aligned}
\end{equation*}
Since $\psi_{b,c}(x)$ is identically zero in a neighborhood of $b+c$ and
identically one in a neighborhood of $b$, this reduces to
\begin{equation*}
\begin{aligned}
&\sum_{n=b}^{b+c} {F}_{\alpha}^{\epsilon}(nd)\psi_{b,c}(n) -
\int_{b}^{b+c} {F}_{\alpha}^{\epsilon}(xd)\psi_{b,c}(x) \,dx \\
=&{F}_{\alpha}^{\epsilon}(bd)
-\sum_{\gamma=1}^{s} \frac{(-1)^{\gamma}B_{\gamma}}{\gamma!} \left.
\frac{d^{\gamma-1}}{dx^{\gamma-1}}
\left({F}_{\alpha}^{\epsilon}(xd)\right)\right|_{x=b} \\
&+\frac{(-1)^{s-1}}{s!} \int_{b}^{b+c} B_{s}(x-[x])\left[\left.\frac{d^s}{dt^s}
\left({F}_{\alpha}^{\epsilon}(td)\psi_{b,c}(t)\right)
\right|_{t=x}\right]\, dx \, .
\end{aligned}
\end{equation*}
We expand the derivatives as before and define
\begin{equation}
\begin{aligned}
\mathcal{E}_{\alpha,b,c,p,j}^{\epsilon} =
\int_{b}^{b+c}
B_{s}(x-[x])\psi_{b,c}^{(j)}(x)f^{(p)}((k+i\epsilon)xd)
e^{i\alpha xd}\, dx \, ,
\end{aligned}
\end{equation}
\begin{equation}
\begin{aligned}
E_{\alpha,b,c}^{\epsilon}=
&\frac{(-1)^{s-1}}{s!}
\sum_{j=0}^{s} \sum_{j_1=0}^{s-j}\binom{s}{j} \binom{s-j}{j_1} d^{s-j}
(k+i\epsilon)^{s-j-j_1} (i\alpha)^{j_1}
\mathcal{E}_{\alpha,b,c,s-j-j_1,j}^{\epsilon} \, .
\end{aligned}
\end{equation}
Then
\begin{equation}
\label{eq:diff_formula}
\begin{aligned}
&\sum_{n=b}^{b+c} {F}_{\alpha}^{\epsilon}(nd)\psi_{b,c}(n) -
\int_{b}^{b+c} {F}_{\alpha}^{\epsilon}(xd)\psi_{b,c}(x) dx \\
=&{F}_{\alpha}^{\epsilon}(bd)
-\sum_{\gamma=1}^{s} \frac{(-1)^{\gamma}B_{\gamma}}{\gamma!} \left.
\frac{d^{\gamma-1}}{dx^{\gamma-1}}
\left({F}_{\alpha}^{\epsilon}(xd)\right)\right|_{x=b} 
+E_{\alpha,b,c}^{\epsilon} \, .
\end{aligned}
\end{equation}
The limit of each term as $\epsilon\rightarrow 0^+$ exists
(trivially for the first two terms, by the dominated convergence theorem for
the third):
\begin{equation*}
\begin{aligned}
&\lim_{\epsilon\rightarrow 0^+}
\left[\sum_{n=b}^{b+c} {F}_{\alpha}^{\epsilon}(nd)\psi_{b,c}(n) -
\int_{b}^{b+c} {F}_{\alpha}^{\epsilon}(xd)\psi_{b,c}(x) dx\right] \\
=&{F}_{\alpha}(bd)
-\sum_{\gamma=1}^{s} \frac{(-1)^{\gamma}B_{\gamma}}{\gamma!} \left.
\frac{d^{\gamma-1}}{dx^{\gamma-1}}
\left({F}_{\alpha}(xd)\right)\right|_{x=b} 
+\lim_{\epsilon \rightarrow 0^+} E_{\alpha,b,c}^{\epsilon} \, .
\end{aligned}
\end{equation*}
Notice that the first two terms are fully independent of $c$; to complete
the argument, it suffices to show that
$\lim_{c\rightarrow\infty}\lim_{\epsilon\rightarrow 0^+}
\mathcal{E}_{\alpha,b,c,p,j}^{\epsilon}$
exists (we \emph{do not} claim a limit of zero). Moreover, the
$c\rightarrow\infty$ convergence behavior of the desired limit is fully
determined by the $c\rightarrow\infty$ convergence behavior of
$\lim_{\epsilon\rightarrow 0^+}\mathcal{E}_{\alpha,b,c,p,j}^{\epsilon}$.

As before, we replace $B_{s}(x-[x])$ by its absolutely convergent
Fourier expansion ($s\geq2$) and use Fubini's theorem:
\begin{equation}
\label{eq:diff_formula_2}
\begin{aligned}
\mathcal{E}_{\alpha,b,c,p,j}^{\epsilon} =&
\int_{b}^{b+c}
B_{s}(x-[x])\psi_{b,c}^{(j)}(x)f^{(p)}((k+i\epsilon)xd)
e^{i\alpha xd}dx\\
=&A_{s} \sum_{q \in\mathbb{Z}\backslash\{0\}} q^{-s}
\int_{b}^{b+c}
\psi_{b,c}^{(j)}(x)f^{(p)}((k+i\epsilon)xd)
e^{i\alpha xd} e^{2\pi i q x} dx\\
=&-A_{s}\left[ \delta_{j,0} \sum_{q \in\mathbb{Z}\backslash\{0\}}
q^{-s} G_{\alpha,p,q}^{\epsilon}(b)+
 \sum_{q \in\mathbb{Z}\backslash\{0\}} q^{-s}
\int_{b}^{b+c} \psi_{b,c}^{(j+1)}(x) G_{\alpha,p,q}^{\epsilon}(x) dx\right]\\
=&-A_{s} \delta_{j,0} \sum_{q\in\mathbb{Z}\backslash\{0\}}
q^{-s} G_{\alpha,p,q}^{\epsilon}(b)-
\mathcal{R}_{\alpha,b,c,p,j}^{\epsilon} \, .
\end{aligned}
\end{equation}
Due to the uniform boundedness of $G_{\alpha,p,q}^{\epsilon}(x)$
on $[1,\infty)$ in the simultaneous limit
$\epsilon\rightarrow 0^+$, $|q|\rightarrow\infty$, we may commute the
$\epsilon\rightarrow 0^+$ limit with the sum in the first term:
\begin{equation}
\label{eq:curly_E_limit}
\limsup_{\epsilon\rightarrow 0^+}\left|
\mathcal{E}_{\alpha,b,c,p,j}^{\epsilon}
+A_{s} \delta_{j,0} \sum_{q \in\mathbb{Z}\backslash\{0\}}
q^{-s} G_{\alpha,p,q}(b) \right| = 
\limsup_{\epsilon\rightarrow 0^+}
\left|\mathcal{R}_{\alpha,b,c,p,j}^{\epsilon}\right| \, .
\end{equation}
The right-hand side is
$\mathcal{O}(c^{-j-N-1/2})$ as $c\rightarrow\infty$, for any $N$, as has
already been established.

\hfill $\square$

\section{Informal description of the algorithm}\label{secfast}

Using equations (\ref{eq:algorithm_formula}) and Lemma \ref{mixlemma},
we can conclude the existence of 
$\lim_{\epsilon \rightarrow 0^+} {S}^\epsilon (\alpha)$
and obtain an explicit formula for this quantity 
(for a fixed $s\geq 2$):
\begin{equation}
\label{eq:exact_sum_eq}
\begin{aligned}
\lim_{\epsilon \rightarrow 0^+}
{S}^{\epsilon}(\alpha)=&\sum_{n=1}^{b}{F}_{\alpha} (nd)
-\sum_{\gamma=1}^{s} \frac{(-1)^{\gamma}B_{\gamma}}{\gamma!} \left.
\frac{d^{\gamma-1}}{dx^{\gamma-1}}
\left({F}_{\alpha} (xd)\right)\right|_{x=b} \\ 
&+\frac{1}{d} \int_{0}^{\infty} {F}_{\alpha} 
\left(bd+ i \tau \textnormal{sgn} (k+\alpha) \right)
i \textnormal{sgn} (k+\alpha) d\tau \\
& - \frac{(-1)^{s-1}}{s!} A_s \sum_{j_1=0}^s \binom{s}{j_1} d^s
k^{s-j_1} (i\alpha)^{j_1} \left[ \sum_{q \in \mathbb{Z}\backslash
\left\{0\right\}} q^{-s} G_{\alpha,s-j_1,q} (b) \right] \, .
\end{aligned}
\end{equation}
While this formula is mathematically interesting, it does not provide
a convenient algorithm due to the presence of terms depending 
on $G_{\alpha,p,q} (b)$. Instead, our algorithm approximates this
limit by computing the first four terms of the formula
(\ref{eq:algorithm_formula}) at $\epsilon=0$, for larger and larger
values of $c$. Specifically, the expression is evaluated for increasing
$c$ at fixed $b$ until convergence is obtained with the desired number
of digits. This procedure is formally justified due to
Lemma \ref{convlemma} and Lemma \ref{mixlemma}.

It is, perhaps, of interest to note that
the conditional convergence of the lattice sums is discussed in the 
literature (see, for example \cite{Linton2010}), but not typically studied in
terms of an explicit limit with vanishing dissipation. An exception is
\cite{Dienstfrey2001}.

\subsection{Wood anomalies}

As noted in the introduction, when
$(k \pm \alpha)d = 2\pi n$, where $n \in \mathbb{Z}$,
the lattice sums considered in this paper may actually diverge.
Such singularities are known as {\em Wood's anomalies} and correspond
to a kind of physical resonance in the underlying system. 
Using (\ref{eq:exact_sum_eq}) and (\ref{eq:lim_G_eps}), 
we expect to see $|{S}(\alpha)| \lesssim |\delta k|^{-\frac{1}{2}}$ 
for sums involving $H_\ell$ and $|{S}(\alpha)| \lesssim
\log \frac{1}{|\delta k|}$ for sums involving $h_\ell$
as $|\delta k| \rightarrow 0$. Here, $|\delta k|$ is the distance 
(in the wavenumber variable $k$) to the closest Wood anomaly. 
Two dimensional Wood anomalies (see \eqref{wood2d}), are more complicated
to characterize and will be considered at a later date.

\section{Numerical Validation}\label{sec-numerical}

Our fast summation algorithm for one-dimensional lattice sums, as
previously indicated, is based upon the formula 
(\ref{eq:algorithm_formula}),
\begin{equation*}
\begin{aligned}
{S}^{\epsilon}(\alpha)= & \sum_{n>0}
{F}_\alpha^\epsilon (n d) = \sum_{n > 0}
f ( (k+i \epsilon) n d) e^{i \alpha n d} \\
=&\sum_{n=1}^{b-1}{F}_{\alpha}^{\epsilon}(nd)
+\sum_{n=b}^{b+c}{F}_{\alpha}^{\epsilon}(nd)\psi_{b,c}(n)
-\int_{b}^{b+c}{F}_{\alpha}^{\epsilon}(xd)\psi_{b,c}(x)\, dx\\
&+\frac{1}{d} \int_{0}^{\infty} {F}_{\alpha}^{\epsilon}
\left(bd+\frac{\epsilon+i(k+\alpha)}{\sqrt{\epsilon^2+(k+\alpha)^2}}\tau\right)
\frac{\epsilon+i(k+\alpha)}{\sqrt{\epsilon^2+(k+\alpha)^2}}\, d\tau
+R_{s,b,c}^{\epsilon}(\alpha) \, .
\end{aligned}
\end{equation*}
Using Lemma \ref{convlemma} and Lemma \ref{mixlemma} which have been
proven in the preceding sections, if $c$ is sufficiently large then
\begin{equation}
\label{eq:alg_expr}
\begin{aligned}
\lim_{\epsilon \rightarrow 0^+} {S}^\epsilon(\alpha) \approx
&\sum_{n=1}^{b-1}{F}_{\alpha} (nd)
+\sum_{n=b}^{b+c}{F}_{\alpha} (nd)\psi_{b,c}(n)
-\int_{b}^{b+c}{F}_{\alpha} (xd)\psi_{b,c}(x)\, dx\\
&+\frac{1}{d} \int_{0}^{\infty} {F}_{\alpha}
\left(bd+i \tau \textnormal{sgn} \left( k+\alpha \right)\right)
i \, \textnormal{sgn} \left( k+\alpha \right)
\, d\tau \, .
\end{aligned}
\end{equation}
Moreover, for any fixed cutoff function $\psi$ as 
in the statement of Theorem 3.1, the error associated with this
approximation vanishes more quickly than any power of $c$, as
$c\rightarrow \infty$.

The fast algorithm proceeds as follows: First, we fix (once and for all)
a smooth cut-off function $\psi$ as in the statement of Theorem 3.1.
We also fix a positive integer value for the parameter $b$, which does
not play an essential role in the analysis. Second, we evaluate the
right-hand side of (\ref{eq:alg_expr}) for integer values of $c$,
doubling $c$ at each iteration until convergence is obtained.

To test the speed and accuracy of our algorithm,
we consider the cases where
$f$ is the Hankel function of the first kind $H_\ell^{(1)}$ or
the spherical Hankel function of the first kind $h_\ell^{(1)}$. With
a Fortran implementation on a
1.3GHz Intel Core M processor, we found that the first
120 lattice sums $S_l$ in (\ref{Sndef}) were computed in
1.2 milliseconds with 
$7$ digits of accuracy for $k = 1$, with a unit cell of length 
$d = 1$ and $\alpha= 0.4$.
For $k = 10$, 1.5 milliseconds were required and for 
$k = 100$, 300 lattice sums were computed in 
2 milliseconds.
For validation, the sums ${S}^\epsilon(\alpha)$ were computed 
directly for $\epsilon > 0$, and the limit $\epsilon \rightarrow 0^+$ 
was determined numerically using fourth-order Richardson extrapolation
starting at $\epsilon = 10^{-3}$, with an estimated accuracy of twelve
digits.

In Fig. \ref{scanfig}, we scan a range of frequencies $k$ with the phase
$\alpha$ set to either $0$ or $0.4$. The blowup of
$S_0$, defined by (\ref{Sndef}) is clearly visible.
To investigate the behavior of our numerical method near Wood anomalies,
we set parameters to 
$d = 1$, $\alpha = \frac{\pi}{4}$, and $k = \frac{7 \pi}{4} + \eta$,
where $0 < \eta \leq 1$, and consider the lattice sums 
$S_l$ involving $H_\ell$ and $h_\ell$, respectively.
In Fig. \ref{woodfig}, we plot $S_l$ and $S^{grating}_{l,0}$, defined
as the one-dimensional grating sum component of (\ref{Snmdef}), 
for $l = 0,1,2,3$.
In the first case, there is a clear power-law singularity with
the correct exponent. In the latter case, involving $h_\ell$,
the blowup is only logarithmic and harder to fit precisely.

\begin{figure}[h]
\centering
\includegraphics[width=0.6\textwidth]{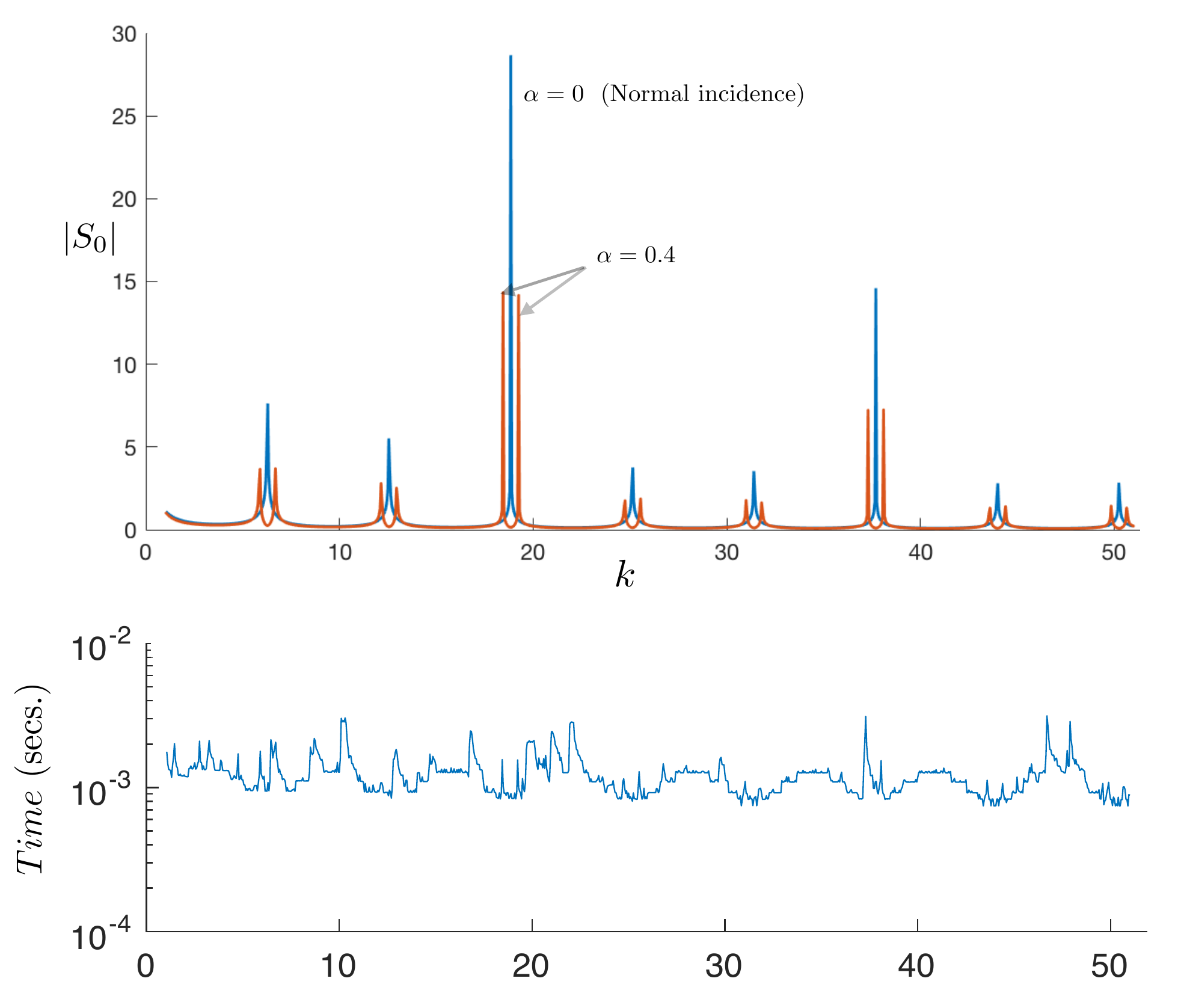}
\caption{(Top) The periodic blowup of the lattice sums
$S_0(\alpha)$, defined by (\ref{Sndef}) 
as a function of frequency $k$, with $\alpha$ set to either $0$ or $0.4$ 
and $d=1$.
(Bottom) Time required for lattice sums with $\alpha = 0.4$ at eack $k$.}
\label{scanfig}
\end{figure}

\begin{figure}[h]
\centering
\includegraphics[width=0.4\textwidth]{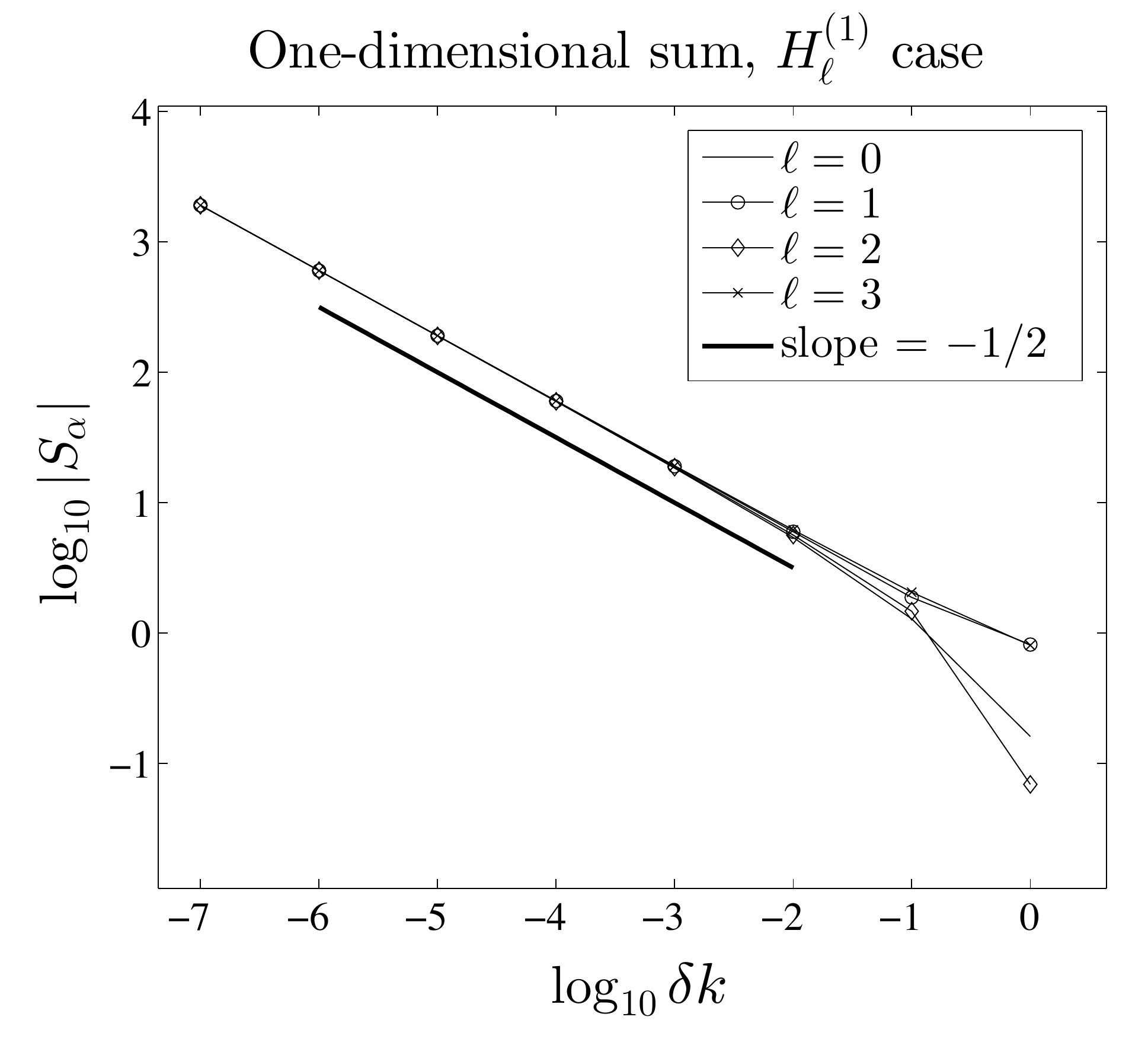}\ 
\includegraphics[width=0.43\textwidth]{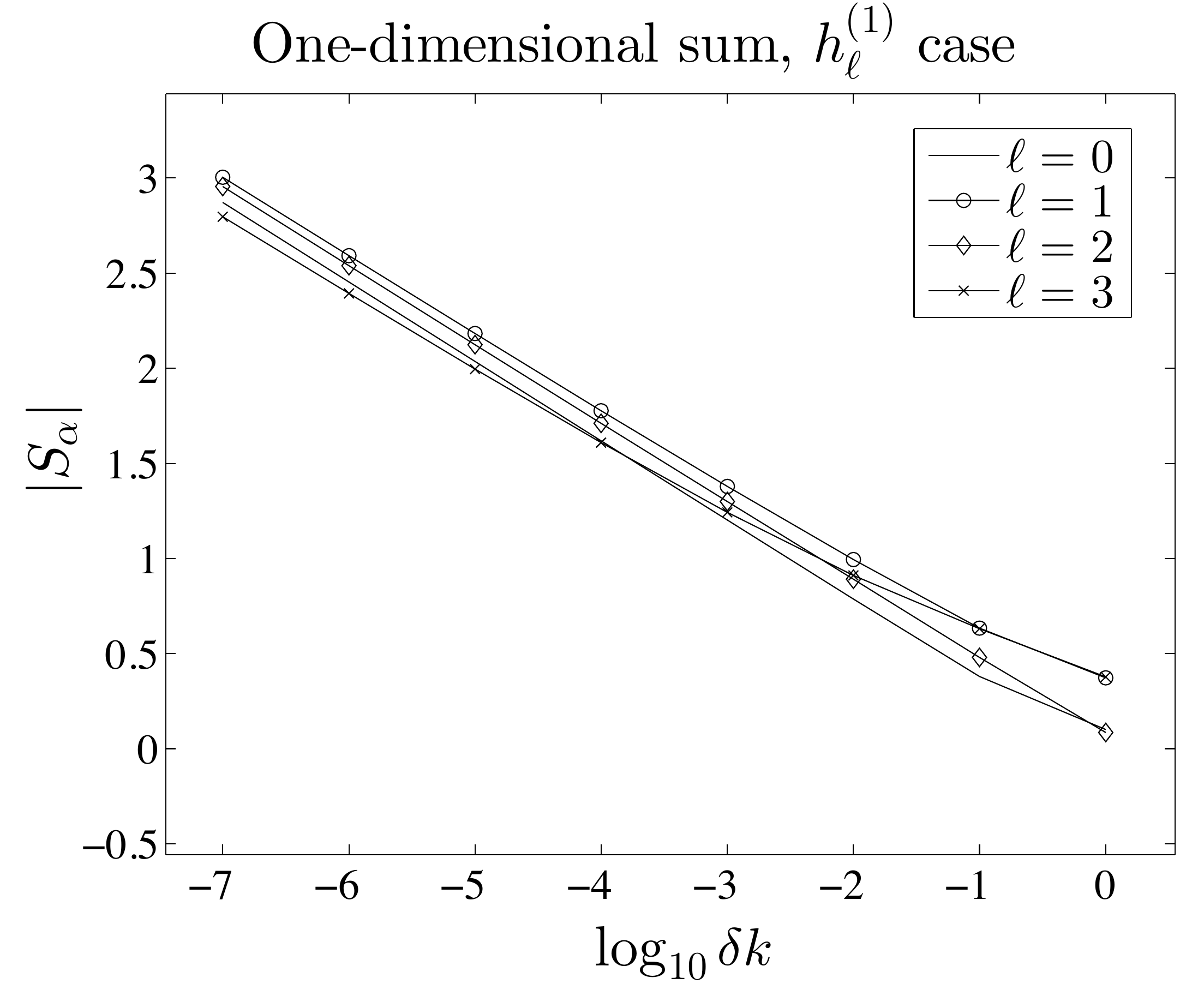}
\caption{(l) Power-law blow-up of $S_l(\alpha)$ near a Wood's anomaly
when $f=H^{(1)}_\ell$ with 
parameters $d=1$, $\alpha = \frac{\pi}{4}$, 
$k=\frac{7\pi}{4}+ \delta k$.
(r) Blow-up of the lattice sums $S^{grating}_{l,0}(\alpha)$ near 
a Wood's anomaly when $f=h^{(1)}_\ell$, with the same parameters.
Note that on the left, the $y$-axis is on a logarithmic
scale, while on the right, it is on a linear scale.\label{woodfig} }
\end{figure}

\section{Conclusions} \label{concl}

We have described a general approach for the numerical evaluation of 
one-dimensonal lattice sums which play an important role in 
diffraction and wave propagation problems in both
two and three dimensions. Indeed, it is often possible to efficiently reduce
higher-dimensional sums to their one-dimensional counterparts as 
discussed in section \ref{latticeredux}. Our 
algorithm achieves super-algebraic convergence rates and is able
to evaluate lattice sums accurately and efficiently.
Moreover, our estimates supply an interesting analytic 
interpretation of Wood's anomalies - physical resonances that occur
when $(k \pm \alpha)d = 2 \pi n $ for integer $n$ - which cause the lattice sums
to diverge \cite{LintonThompson2007,Linton2010}. 
In particular, the formulas (\ref{eq:lim_G_eps}) and
(\ref{eq:exact_sum_eq}) allow us to directly estimate the type of blow-up
one should expect to see. 

We believe that
higher dimensional Wood anomalies can be analyzed by
coupling our method with lattice reduction techniques, and we 
will report on such work at a later date.
Finally, it is worth noting that lattice sums can be avoided altogether.
Quasiperiodic boundary conditions can be imposed, for example,
using layer potentials 
\cite{BarnettGreengard2010,BarnettGreengard2011,Gillman2013}, fundamental
solutions \cite{Cho2014}, or spherical harmonics \cite{acper}
to enforce the conditions explicitly.
Which approach is more efficient will depend, we suspect, on the ambient
dimension, the frequency, and on the aspect ratio of the unit cell.
Both this question and a comparison with other fast algorithms, such as those
discussed in 
\cite{Kurkcu2009,Linton2010,LintonThompson2007,LintonThompson2009,
Maystre2012,McPhedranGrubits2000,Moroz2001,Moroz2006,
Twersky1961, YasumotoYoshitomi1999} remain to be explored.

\bibliography{lattice_paper_bib}{}
\bibliographystyle{plain}
\end{document}